\documentclass[10pt,reqno]{amsart}

\usepackage{amsmath,amscd,amssymb,amsthm}
\usepackage{tikz-cd}
\usepackage{comment}

\makeatletter
\def\th@plain{%
  \thm@notefont{}
  \itshape 
}
\def\th@definition{%
  \thm@notefont{}
  \normalfont 
}
\makeatother

\usepackage[pdftex,hyperindex]{hyperref}

\usepackage{enumerate}
\usepackage{todonotes}
\usepackage{color}

\newtheorem{proposition}{Proposition}[section]
\newtheorem{lemma}[proposition]{Lemma}
\newtheorem{theorem}[proposition]{Theorem}
\newtheorem{corollary}[proposition]{Corollary}

\theoremstyle{definition}
\newtheorem{remark}[proposition]{Remark}
\newtheorem{definition}[proposition]{Definition}
\newtheorem{example}[proposition]{Example}

\numberwithin{equation}{section} 

\newcommand{\N}{\mathbb{N}}
\newcommand{\R}{\mathbb{R}}
\newcommand{\C}{\mathbb{C}}

\newcommand{\pr}{\mathbb{P}}

\newcommand{\ddbar}{\partial\bar{\partial}}

\newcommand{\B}{\mathcal{B}}
\newcommand{\scL}{\mathcal{L}}

\newcommand{\scH}{\mathcal{H}}
\newcommand{\scM}{\mathcal{M}}

\newcommand{\scX}{\mathcal{X}}
\newcommand{\X}{\mathcal{X}}
\newcommand{\Y}{\mathcal{Y}}

\newcommand{\scO}{\mathcal{O}}

\newcommand{\scB}{\mathcal{B}}
\newcommand{\scY}{\mathcal{Y}}

\newcommand{\scA}{\mathcal{A}}

\newcommand{\scS}{\mathcal{S}}

\newcommand{\scP}{\mathcal{P}}

\newcommand{\co}{\Omega}
\renewcommand{\L}{\mathcal{L}}

\DeclareMathOperator{\Aut}{Aut}

\DeclareMathOperator{\Ric}{Ric}

\DeclareMathOperator{\DF}{DF}

\DeclareMathOperator{\Pic}{Pic}

\DeclareMathOperator{\tr}{tr}

\DeclareMathOperator{\Lie}{Lie}

\title[Relative K-stability for K\"ahler manifolds]{Relative K-stability for K\"ahler manifolds}

\author[Ruadha\'i Dervan]{Ruadha\'i Dervan}
\address{Ruadha\'i Dervan, D\'epartement de Math\'ematique, Universit\'e Libre de Bruxelles and Department of Pure Mathematics and Mathematical Statistics, University of Cambridge.}
\email{R.Dervan@dpmms.cam.ac.uk}

\begin{document}

\begin{abstract} 

We study the existence of extremal K\"ahler metrics on K\"ahler manifolds. After introducing a notion of relative K-stability for K\"ahler manifolds, we prove that K\"ahler manifolds admitting extremal K\"ahler metrics are relatively K-stable. Along the way, we prove a general $L^p$ lower bound on the Calabi functional involving test configurations and their associated numerical invariants, answering a question of Donaldson.

When the K\"ahler manifold is projective, our definition of relative K-stability is stronger than the usual definition given by Sz\'ekelyhidi. In particular our result strengthens the known results in the projective case (even for constant scalar curvature K\"ahler metrics), and rules out a well known counterexample to the ``na\"ive'' version of the Yau-Tian-Donaldson conjecture in this setting. 

\end{abstract}

\maketitle

\section{Introduction}

In 1982, Calabi posed the problem of finding an extremal metric in a given K\"ahler class on a K\"ahler manifold \cite{EC}. Extremal metrics are K\"ahler metrics $\omega$ whose scalar curvature $S(\omega)$ satisfies $$\bar\partial \nabla^{1,0}S(\omega) = 0,$$ i.e. the $(1,0)$-part of the gradient of $S(\omega)$ is a holomorphic vector field. When they exist, extremal metrics give a \emph{canonical} choice of K\"ahler metric in their class. However, a given K\"ahler manifold may not admit an extremal metric in certain K\"ahler classes \cite{ACGTF}, or even in any K\"ahler class at all \cite{MC}. A fundamental question in K\"ahler geometry is therefore to characterise the K\"ahler manifolds which admit extremal metrics.

An important special case of extremal metrics are constant scalar curvature K\"ahler (cscK) metrics, and hence K\"ahler-Einstein metrics. Suppose now that the K\"ahler manifold is a smooth projective variety $X$, and the K\"ahler class is the first Chern class of an ample line bundle $L$. In this case, a deep conjecture of Yau-Tian-Donaldson states that $(X,L)$ admits a cscK metric if and only if it satisfies an algebro-geometric condition called \emph{K-stability} \cite{STY,GT,SD2}. This conjecture was extended to the setting of extremal K\"ahler metrics by Sz\'ekelyhidi, who conjectured that $(X,L)$ admits an extremal K\"ahler metric if and only if it is \emph{relatively K-stable} \cite{GS}. The importance of K-stability has been underlined by Chen-Donaldson-Sun's proof that K-stability is equivalent to the existence of a K\"ahler-Einstein metric on Fano manifolds \cite{CDS}. 

One direction of these conjectures is now essentially proven, namely the existence of a cscK (resp. extremal) metric  on a projective variety implies K-stability \cite{JS,SD,BDL,CS} (resp. relative K-stability \cite{StSz}). 

In \cite{DR}, we defined a notion of K-stability for K\"ahler manifolds, and proved:

\begin{theorem}\label{introdiscrete}\cite{DR} Suppose $(X,[\omega])$ is a K\"ahler manifold with discrete automorphism group. If $(X,[\omega])$ admits an cscK metric, then it is K-stable.\end{theorem} \noindent This was independently proven by Sj\"ostr\"om Dyrefelt \cite{ZSD}.

In the present article, we define a notion of relative K-stability for K\"ahler manifolds using K\"ahler techniques. Our main result is:

\begin{theorem}\label{intromaintheorem} If $(X,[\omega])$ admits an extremal metric, then it is relatively K-stable.\end{theorem}

Strictly speaking, our definition of relative K-stability should be called ``K-polystability relative to a maximal torus''. The following is therefore an immediate corollary.

\begin{corollary}\label{introcor} If $(X,[\omega])$ admits an cscK metric, then it is equivariantly K-polystable. \end{corollary} \noindent This extends the results of \cite{DR,ZSD} to the setting of K\"ahler manifolds admitting automorphisms. 

Perhaps the most important aspect of Theorem \ref{intromaintheorem} (and Corollary \ref{introcor}) is that when our K\"ahler manifold is projective, our definition of relative K-stability is \emph{stronger} than the definition given by Sz\'ekelyhidi \cite{GS}. Our results therefore strength\-en the previously known results in the projective case, both for extremal metrics and cscK metrics. Roughly speaking, K-stability of projective varieties involves a set of auxiliary varieties $\scX$, called test configurations, together with line bundles $\scL$. A very influential example of \cite{ACGTF} strongly suggests that for relative K-stability to be equivalent to the existence of an extremal metric, one needs to allow test configurations together with \emph{irrational} line bundles, that is, formal tensor powers of line bundles with $\R$-coefficients. As the first Chern class of such an object makes sense (as the sum of the first Chern classes of the combination), our K\"ahler theory of relative K-stability naturally incorporates these objects. In particular, the example of \cite{ACGTF} is \emph{not} relatively K-stable in our K\"ahler sense. This is the first time this example has been ruled out (it would also be ruled out by relative analogues of other stronger notions of K-stability \cite{GS2,RD,BHJ1}, however it appears to be a very challenging problem to prove that the existence of an extremal metric implies these notions). Although our stronger K\"ahler notion of relative K-stability rules out the phenomenon described in \cite{ACGTF}, it may perhaps be too optimistic to conjecture that it implies the existence of an extremal metric; we discuss this further in Remark \ref{YTD-remark}.

A key part of our proof of Theorem \ref{intromaintheorem} is to prove the following generalisation of Donaldson's lower bound on the Calabi functional \cite{SD}, which new for general $p$ even when $X$ is projective:

\begin{theorem}\label{introcalabi} We have \[\inf_{\omega\in [\omega]} \|S(\omega)-\hat S\|_p \geq -\sup_{(\X,\scA)\in \mathbb{T}} \frac{\DF(\X,\scA)}{\|(\X,\scA)\|_q}.\]\end{theorem}

\noindent Here $\mathbb{T}$ denotes the set of test configurations for $(X,[\omega])$, $(p,q)$ is an arbitrary H\"older conjugate pair, $\hat S$ is the average scalar curvature of $\omega$ and $\DF(\X,\scA)$ (resp. ${\|(\X,\scA)\|_q}$) denotes an important numerical invariant called the Donaldson-Futaki invariant of the test configuration $(\X,\scA)$ (resp. the $L^q$-norm of  $(\X,\scA)$). Donaldson proved the above result when $X$ is projective and $[\omega] = c_1(L)$ for an ample line bundle $L$, provided $q$ is an even integer \cite{SD}. Donaldson also asked whether the above result holds for general H\"older conjugate pairs $(p,q)$; this answers his question. When $(X,[\omega])$ admits a cscK metric, this implies $(X,[\omega])$ is K-semistable, giving a slightly different proof of the main results of \cite{DR,ZSD}. However, the main interest in Theorem \ref{introcalabi} is in the case that $(X,[\omega])$ does \emph{not} admit a cscK metric. We conjecture equality holds in Theorem \ref{introcalabi} when $p=q=2$, by analogy with Donaldson's conjecture in the projective case \cite{SD}. Remark that equality does not hold for other $(p,q)$ even when $(X,[\omega])$ admits an extremal metric that is not cscK.

Analogues of Theorem \ref{introcalabi} can be proven using similar methods for twisted cscK metrics and the J-flow. Here one replaces the left hand side with $\|S(\omega) - \Lambda_{\alpha}\omega -c_1\|_p$ for twisted cscK metrics, and with $\|\Lambda_{\alpha}\omega -c_2\|_p$ for the J-flow (where $c_1,c_2$ are the appropriate topological constants and $\alpha$ is an auxiliary K\"ahler metric in an arbitrary K\"ahler class), and replaces the right hand side with the corresponding numerical invariants \cite{RD,LS} (for the J-flow one should use the numerical invariants as formulated in \cite[Section 4.2]{DK} rather than the original formulation in \cite{LS}, as in \cite[Section 6]{DR}). In the projective case, these results were proven in \cite{RD,LS} for $p=q=2$.

The techniques we develop will also clarify some aspects of K\"ahler K-stability and will lead to results which are of independent interest. Firstly, we will relate the norms of test configurations and their corresponding geodesics. Using this, we will also be able to characterise the trivial K\"ahler test configurations, clarifying the definition of K-stability given in \cite{DR} (by the pathological examples of Li-Xu \cite{LX}, it is a rather subtle problem to understand what it means for a test configuration to be trivial, even in the projective case). 

\begin{theorem}\label{intro-norms} Let $(\scX,\scA)$ be a test configuration. The following are equivalent:
\begin{itemize}
\item[(i)] the $L^p$-norm $\|(\scX,\scA)\|_p$ vanishes for some $p$,
\item[(ii)] the $L^p$-norm $\|(\scX,\scA)\|_p$ vanishes for all $p$,
\item[(iii)] the minimum norm $\|(\scX,\scA)\|_m$ vanishes,
\end{itemize}
If $\X$ is smooth, then the $L^p$-norm of $(\scX,\scA)$ equals the $L^p$-norm of the associated geodesic, and hence these are also equivalent to the geodesic associated to $(\scX,\scA)$ being trivial. Finally, the $L^1$-norm of a test configuration is Lipschitz equivalent to the minimum norm. 
\end{theorem} \noindent The minimum norm \cite{RD} is also called the ``non-Archimedean J-functional'' \cite{BHJ1}. It follows that uniform K-stability with respect to the $L^1$-norm (in the sense of \cite{GS4}) is equivalent to uniform K-stability with respect to the minimum norm (in the sense of \cite{RD,BHJ1}) in this general K\"ahler setting, extending the corresponding projective result \cite{BHJ1} (the advantage of the minimum norm being that it is closely related to analytic functionals and intersection theory). The result relating the norm of a test configuration to the norm of the associated geodesic is due to Hisamoto in the projective case (without any smoothness assumption), who proved this by relating both quantities to an associated Duistermatt-Heckman measure \cite{TH}; we give a more direct proof which applies in the K\"ahler setting. The remaining results in Theorem \ref{intro-norms} extend to the K\"ahler setting, and give somewhat different proofs of, some of the main results of \cite{RD,BHJ1}, which were proven in the projective case. 

\subsection{Comparison with other work}

Although this article is essentially a sequel to \cite{DR,ZSD}, where Theorem \ref{introdiscrete} was proven, the techniques used are very different. In \cite{DR,ZSD} the main theme was to differentiate energy functionals on the space of K\"ahler metrics along certain paths induced by test configurations. To prove Theorems \ref{intromaintheorem} and \ref{introcalabi} we instead use a combination of the results obtained in \cite{DR,ZSD} together with a delicate use hamiltonian geometry to obtain precise information about how the invariants of a test configuration (such as the norm) change when one perturbs the test configuration.

Similarly, although our result recovers and extends the corresponding projective results \cite{SD,StSz}, the arguments involve very different techniques. Donaldson's proof the lower bound on the Calabi functional for projective varieties involves the use of Bergman kernels to reduce to a finite dimensional problem \cite{SD}. Donaldson then uses a convexity result  in finite dimensions, arising from geometric invariant theory, to obtain his result. This avoids the use of the convexity of the Mabuchi functional \cite{BB}, which was not available at the time. Stoppa-Sz\'ekelyhidi similarly use embeddings into projective space and finite dimensional geometric invariant theory \cite{StSz} to prove their projective analogue of Theorem \ref{intromaintheorem}. As these are not available to us in the K\"ahler setting, we instead use analytic arguments using geodesics.

In forthcoming work, Sj\"ostr\"om Dyrefelt independently proves Corollary \ref{introcor}, regarding cscK metrics, using very different techniques \cite{ZSD2}. Interestingly, Sj\"ostr\"om Dyrefelt proves the stronger result that the existence of a cscK metric implies K-polystability rather than the equivariant K-polystability that we prove. As his proof uses deep analytic results on the Mabuchi functional for which the corresponding results for extremal metrics are not known, it would be difficult to use his techniques to prove Theorem \ref{intromaintheorem}. Sz\'ekelyhidi has proven a weak version of Theorem \ref{intromaintheorem}, for test configurations with smooth central fibre \cite[Section 4.1]{GS6}. Most test configurations of interest (for example those arising from deformation to the normal cone) will be highly singular, for example they will typically not even have irreducible central fibre. 

Analogues of our lower bound on the Calabi functional using the Mabuchi functional have been proven before \cite{TH,XC3}, being suggested first by Chen and Donaldson \cite[p3]{SD}. Chen proved a result similar to Theorem \ref{introcalabi} along smooth geodesics using the $\Psi$-invariant, which is defined as the limit derivative of the Mabuchi functional along a smooth geodesic \cite{XC3}. As geodesics are rarely smooth \cite{LV}, it is an essential point in our argument to work with general (singular) geodesics.  Hisamoto uses similar ideas and the Ding functional to prove an analogue of Theorem \ref{introcalabi} in the Fano case, with $[\omega]=c_1(X)$ and a different functional replacing the Calabi functional \cite{TH}. The first appearance of a result similar to Theorem \ref{introcalabi} was in the seminal work of Atiyah-Bott on Yang-Mills theory \cite{AB}. Essentially these results arise from the point of view of moment maps and geometric invariant theory; see for example the survey \cite{GRS}. Our lower bound is closely related to this moment map theory, and our proof is an infinite dimensional analogue of the usual proof in geometric invariant theory.

\subsection{Outline}

In Section \ref{relative-projective-sec} we discuss relative K-stability for projective varieties, following Sz\'ekelyhidi \cite{GS}. We then define a notion of relative K-stability for K\"ahler manifolds in Section \ref{rel-kstab-kahler-sec}. Section \ref{prelim-mabuchi-sec} contains preliminaries on the Mabuchi functional and geodesics, which are then used in Section \ref{calabi-sec} to prove Theorem \ref{introcalabi}. In Section \ref{norms-sec} we use the techniques developed to prove Theorem \ref{intro-norms}. Section \ref{relative-proof-sec} contains the proofs of Theorem \ref{intromaintheorem} and Corollary \ref{introcor}.

\vspace{4mm}

\noindent {\bf Notation and conventions:} We work throughout over the complex numbers. For notational convenience we ignore certain dimensional constants and factors of $2\pi$ which play no important role, so for example the right hand side of Theorem \ref{introcalabi} should have a factor which is a function only of $\dim X$. We use the language of K\"ahler geometry on analytic spaces, for this we refer to \cite{JPD} for an introduction.  For closed $(1,1)$-forms $\co_0,\hdots,\co_n$ on an $(n+1)$-dimensional analytic space $\X$, we often denote $$[\co_0].\hdots .[\co_n]= \int_{\X}\co_0\wedge\hdots\wedge \co_n.$$ We also sometimes call this an intersection number, borrowing the terminology of the case that $\co_i \in c_1(\L_i)$ for some line bundles $\L_i\to \X$. For an analytic space $\X\to\C$ (or $\pr^1$), we will denote $\X_t$ the fibre over $t\in \C$ (or $t\in\pr^1$). Likewise for a $(1,1)$-form $\co$ on $\X$, its restriction to a fibre will be denoted $\co_t$.

\vspace{4mm}

\noindent {\bf Acknowledgements:} I would like to thank Vestislav Apostolov, Joel Fine, Julius Ross, Zakarias Sj\"ostr\"om Dyrefelt, Jacopo Stoppa, G\'abor Sz\'ekelyhidi and Xiaowei Wang for helpful discussions. This work was mostly done while the author was a Fondation Wiener Anspach scholar at the Universit\'e libre de Bruxelles.

\section{Relative K-stability}
\subsection{Relative K-stability for projective varieties}\label{relative-projective-sec}

Let $(X,L)$ be a normal polarised variety, i.e. $L\to X$ is an ample line bundle. In this section, following Sz\'ekelyhidi \cite{GS}, we briefly recall what it means for $(X,L)$ to be relatively K-stable. First of all we define a set of degenerations of $(X,L)$, called \emph{test configurations}.

\begin{definition}\cite[Definition 2.1.1]{SD2} A \emph{test configuration} $(\X,\L)$ for $(X,L)$ is a normal variety $\X$ together with
\begin{itemize} 
\item[(i)] a flat (i.e. surjective) morphism $\pi: \scX \to \C$,
\item[(ii)] a $\C^*$-action $\alpha$ on $\scX$ covering the natural action on $\C$,
\item[(iii)] and an equivariant relatively ample line bundle $\scL$ on $\scX$, 
\end{itemize}
such that the fibre $(\scX_t,\scL_t)$ over $t$ is isomorphic to $(X,L^r)$ for one, and hence all, $t \in \C^*$ and for some $r>0$. We call $r$ the \emph{exponent} of  $(\X,\L)$.
\end{definition}

We will now extract numerical invariants from this data. The $\C^*$-action $\alpha$ induces a $\C^*$-action on the central fibre $(\X_0,\L_0)$, and hence on $H^0(\X_0,\L_0^k)$ for all $k$. Denote by $A_k$ the infinitesimal generator of this $\C^*$-action. Suppose moreover that $\X$ admits a \emph{vertical} $\C^*$-action $\beta$ lifting to $\scL$, i.e. $\beta$ fixes the fibres of $\pi$. Let $B_k$ be the infinitesimal generator of the action of $\beta$ on $H^0(\X_0,\L_0^k)$. From this data, we define polynomials for $k\gg 0$ as follows: \begin{align*} \dim H^0(\X_0,\L_0^k) &= a_0k^n+a_1k^{n-1}+O(k^{n-2}), \\ \tr (A_k) &= b_0 k^{n+1} + b_1 k^n + O(k^{n-1}), \\ \tr(B_k) &= c_0k^{n+1}+O(k^{n}), \\ \tr(A_k^2) &= d_0k^{n+2} + O(k^{n+1}), \\ \tr(A_k B_k) &= e_0 k^{n+2} + O(k^{n+1}).\end{align*}

Remark that by flatness, for $k\gg 0$ the dimension $\dim H^0(\X_0,\L_0^k)$ equals the Hilbert polynomial of $(X,L^r)$. That the other polynomials are indeed polynomials for $k\gg 0$ follows from equivariant Riemann Roch and its variants; one way of proving this is to use \cite[Section 5.1]{SD}.  

From these polynomials, we define various numerical invariants associated to the test configuration. The most important is the Donaldson-Futaki invariant.
 
\begin{definition}\cite{SD2} We define the \emph{Donaldson-Futaki invariant} of $(\X,\L)$ to be $$\DF(\X,\L) = \frac{b_0a_1 - b_1a_0}{a_0}.$$ \end{definition}

\begin{example} If $(X,L)$ admits a $\C^*$-action $\beta$, one obtains a \emph{product test configuration} by taking the induced $\C^*$-action on $(X\times \C, L)$. The Donaldson-Futaki invariant of such test configurations was introduced by Futaki (using holomorphic vector fields), and we denote it in this case by $F(\beta)$.\end{example}

Next are the norm and inner product.

\begin{definition}\label{inner-prods-norms-RR}\cite{SD,GS} We define the $L^2$\emph{-norm} $\|(\X,\L)\|_2$ of $(\X,\L)$ to be $$\|(\X,\L)\|^2_2 = \frac{d_0a_0 - b_0^2}{a_0}.$$ Similarly we define the \emph{inner product} of the $\C^*$-actions $\alpha$ and $\beta$ to be $$\langle \alpha, \beta \rangle = \frac{e_0a_0-b_0c_0}{a_0}.$$\end{definition}

\begin{definition} Let $T\subset \Aut(X,L)$ be a torus of automorphisms. We say that a test configuration $(\X,\L)$ is $T$-\emph{invariant} if it admits a vertical torus action which commutes with $\alpha$ and restricts to the usual action of $T$ on $(\X_t,\L_t)$ for all $t\neq 0$.\end{definition}

Pick an orthogonal basis $\beta_1,\hdots, \beta_d$ of $\C^*$-actions generating $T$. We denote \begin{equation}\DF_T(\X,\L) = \DF(\X,\L) - \sum_{i=1}^d \frac{\langle\alpha,\beta_i\rangle}{\langle \beta_i, \beta_i \rangle}F(\beta_i).\end{equation}

\begin{definition}\label{rel-kstab-proj}\cite[Definition 2.2]{GS} We say that $(X,L)$ is \emph{K-stable relative to} $T$ if for all test configurations with $\|(\X,\L)\|_2>0$, we have $\DF_T(\X,\L)>0.$ When $T$ is a \emph{maximal} torus, we simply say that $(X,L)$ is \emph{relatively K-stable}.  For clarity we sometimes call this \emph{projective relative K-stability}. \end{definition}

\begin{remark}\label{rmk:projection}Some remarks on the above definitions are in order:

\begin{itemize}
\item The definition of relative K-stability is motivated by notions of stability for varieties in Mumford's Geometric Invariant Theory, most notably Chow stability and Hilbert stability \cite{SD2,GS}.
\item Our requirement that the norm is positive is to exclude pathological test configurations found by Li-Xu \cite[Section 8.2]{LX}. Their examples have $\X$ non-normal, and normalise to the trivial test configuration. These pathological examples are characterised in \cite{RD,BHJ1} as having norm zero.
\item If  $(\X,\L)$ is \emph{orthogonal to} $T$, i.e. if $\langle\alpha,\beta_i\rangle = 0$ for all $i$, relative K-stability just requires that $\DF(\X,\L)>0$ provided $(\X,\L)$ has positive norm.

\item Note that these definitions only involve the $\C^*$-action on $(\X_0,\L_0)$, hence one can for example similarly define the inner product $\langle \alpha,\alpha\rangle$; clearly this equals the square of the $L^2$-norm $\|(\X,\L)\|_2.$
\item Suppose $T$ is a maximal torus. We then say that $(X,L)$ is \emph{equivariantly K-polystable} if in addition $F(\beta_i) = 0$ for all $i$. This is the notion relevant to constant scalar curvature K\"ahler metrics.

\end{itemize}

\end{remark}

To extend the above definitions to the setting of K\"ahler manifolds, we use another way of representing the above quantities in the projective case. First of all we recall how to ``compactify'' a test configuration.

Let $(\X,\L)$ be a test configuration. Since it is equivariantly isomorphic to the trivial family over $\C\backslash \{0\}$, by gluing in the trivial family around infinity one can \emph{canonically} glue a to give a flat family over $\pr^1$. We call such a test configuration \emph{compactified}, and abuse notation by writing it as $(\X,\L)$. We emphasise that this gluing procedure depends on the $\C^*$-action $\alpha$. For example, the compactifications of the product test configurations for $(\pr^1,\scO_{\pr^1}(1))$ are the Hirzebruch surfaces \cite[Example 2.19]{BHJ1}.

The point of using a compactification is the follow intersection-theoretic formula for the Donaldson-Futaki invariant due to Odaka \cite[Corollary 3.11]{YO} and Wang \cite[Proposition 17]{XW}. For this denote the \emph{slope} of $(X,L)$ as \[\mu(X,L) = \frac{-K_X.L^{n-1}}{L^n}.\]

\begin{proposition}\cite{YO,XW}\label{DF-compact-proj} Let $(\X,\L)$ be a compactified test configuration of exponent $r$. The Donaldson-Futaki invariant of  $(\X,\scL)$ is given as the intersection number
 $$\DF(\X,\scL) := \frac{n}{n+1}\mu(X,L^{\otimes r}) \scL^{n+1} + \scL^n.K_{\X / \pr^1}$$
Here we have written $K_{\X / \pr^1}$ to mean the relative canonical class, and we note that the intersection number $\scL^n.K_{\X}$ makes sense by normality of $\X$.
 \end{proposition}
 
We now give a more analytic description of the inner product and the norms of a test configuration $(\X,\L)$. For this, we equivariantly embed $(\X,\L)$ into projective space, so that $(\X,\L)$ is realised as the closure of the orbit of $(X,L)$ under some one-parameter subgroup, see e.g. \cite[Lemma 2]{SD}. Similarly there is a one-parameter subgroup for the action induced by $\beta$. Write $h_{\alpha}, h_{\beta}$ for the hamiltonians corresponding with respect to the Fubini-Study metric, and let $\hat h_{\alpha},\hat h_{\beta}$ be their average values over $\X_0$ (i.e. $\hat h_{\alpha} = \frac{\int_{\X_0}h_{\alpha}\omega_{FS}^n}{\int_{\X_0}\omega_{FS}^n}$).

\begin{proposition} Denote by $V=\int_X c_1(L)^n$ the volume of $(X,L)$. The inner product $\langle \alpha, \beta \rangle$ is given by $$\langle \alpha, \beta \rangle =\int_{\X_0} (h_{\alpha} - \hat h_{\alpha})(h_{\beta} - \hat h_{\beta})\omega^n_{FS}.$$ Hence the $L^2$-norm of $(\X,\L)$ is given as $$\|(\X,\L)\|^2_2 = \int_{\X_0}(h_{\alpha}-\hat h_{\alpha})^2\omega^n_{FS}.$$\end{proposition}

\begin{proof} This was proven by Donaldson for the norm \cite[Section 5.1]{SD}; a similar proof works for the inner products. \end{proof}

This inner product was introduced by Futaki-Mabuchi \cite{FM} when $\X_0$ is smooth. In general, there is no similar integral formula for the Donaldson-Futaki invariant of $(\X,\L)$ when the central fibre is singular, which forces us to work on the total space $\X$.

This more analytic definition allows us to define, following Donaldson, the $L^p$-norm of a test configuration for general $p$. 

\begin{definition}\cite[p20]{SD} We define the $L^p$\emph{-norm} $\|(\X,\scA)\|_p$ of $(\X,\scL)$ to be $$\|(\X,\scA)\|^p_p =\int_{\X_0} \left|h_{\alpha} - \hat h_{\alpha}\right|^p\omega^n_{FS}.$$ \end{definition}

\begin{remark}This clearly agrees with our previous reformulation for $p=2$. When $p$ is an integer, one can give an equivalent definition of the $L^p$-norm using equivariant Riemann-Roch as in Definition \ref{inner-prods-norms-RR} \cite[Section 5.1]{SD}.\end{remark}

\subsection{Relative K-stability for K\"ahler manifolds}\label{rel-kstab-kahler-sec}

We now introduce a notion of relative K-stability for K\"ahler manifolds, generalising the notion of K-stability defined in \cite{DR,ZSD}. We refer to \cite{JPD,DR} for the background on K\"ahler analytic spaces needed. As relative K-stability is a modification of K-stability, we first recall how to define the objects related to K-stability in the K\"ahler setting.

\begin{definition}\cite{ZSD,DR} A \emph{test configuration} for $(X,[\omega])$ is a normal K\"ahler space $(\X,\scA)$, together with
\begin{enumerate}[(i)]
\item a surjective flat map $\pi: \X\to\C$,
\item a $\C^*$-action on $\X$ covering the usual action on $\C$ such that the class $\scA$ is $\C^*$-invariant and K\"ahler on each fibre,
\item the fibre $(\X_t,[\co_t])$ is isomorphic to $(X,[\omega])$ for all $t\neq 0$.
\end{enumerate}
\end{definition}

\begin{remark} Just as in the projective case, one can glue an arbitrary test configuration to its compactification which admits a map to $\pr^1$. The gluing essentially encodes the $\C^*$-action. We freely interchange between a test configuration and its compactification. When choosing $(1,1)$-forms $\co\in\scA$, we assume that $\co$ extends to a smooth form on the compactification.\end{remark}

Just as in the projective case, we denote the \emph{slope} of a K\"ahler manifold by $$\mu(X,[\omega]) = \frac{c_1(X).[\omega]^{n-1}}{[\omega]^{n}}.$$ 

We now recall the definition of the Donaldson-Futaki invariant of a test configuration, as given in \cite{DR,ZSD}. This definition is motivated by the intersection-theoretic version of the Donaldson-Futaki invariant in the projective case of Definition \ref{DF-compact-proj}. 

\begin{definition}\cite{ZSD,DR} Let $(\X,\scA)$ be a test configuration with $\X$ smooth. We define the \emph{Donaldson-Futaki invariant} of $(\X,\co)$ to be \begin{equation*}
\DF(\X,\scA) := \frac{n}{n+1}\mu(X,[\omega]) \scA^{n+1} - (c_1(\X) - \pi^*c_1(\pr^1)).\scA^n.\label{eq:DFI}\end{equation*}
If $\X$ is singular we take a resolution $p: \scY\to\scX$ and define \begin{equation*}
\DF(\X,\scA) := \frac{n}{n+1}\mu(X,[\omega]) (p^*\scA)^{n+1} - (c_1(\scY) - (\pi\circ p)^*c_1(\pr^1)).(p^*\scA)^n.\end{equation*} \end{definition}

This key point is that this is independent of choice of resolution of singularities \cite{DR}.

\begin{example}\label{ACGTF-example} An important example of a K\"ahler test configuration occurs already when $X$ is projective, and $L$ is an ample line bundle. Then one can have an algebraic total space $\X$, however with an \emph{irrational} polarisation $\L$. That is, $\L$ is a formal $\R$-valued tensor product of line bundles. As the first Chern class of such an object makes sense, its Donaldson-Futaki invariant clearly does also. An example of \cite{ACGTF} strongly suggests that one needs to include such test configurations in the ``correct'' definition of relative K-stability. \end{example}

\begin{remark}When $X$ is projective and $\X_0$ is smooth, a result of Popovici implies that $\X_0$ is Moishezon \cite[Theorem 1.4]{DP}. Since it is also K\"ahler by assumption, it follows that $\X_0$ is itself projective. We expect more generally that even if $\X_0$ is a singular analytic space, provided $X$ is projective then $\X_0$ is a projective scheme. \end{remark} 

We now turn to the norms and inner products of test configurations. The main challenge here is obtaining the right definition of a hamiltonian on a normal analytic space. For this, we recall the following characterisation of the hamiltonian when $\X$ is smooth (see e.g. \cite[Example 4.16]{GS5}).

\begin{lemma}Let $(\X,\scA)$ be a smooth test configuration, and let $\co\in \scA$. Define a smooth family of functions $\varphi(t)$ by $$\alpha(t)^*\co - \co = i\ddbar \varphi_t.$$ Then the function $h_{\alpha} = \alpha(t)_*\dot\varphi(t)$ is a hamiltonian for the $S^1$-action induced from $\alpha(t)$ with respect to $\co$. In particular, $h_{\alpha}$ is independent of $t$. \end{lemma}

We simply take this to be our definition of the hamiltonian when $\X$ is singular.

\begin{definition}\label{hamiltonian-def} Let $(\X,\scA)$ be a  (not necessarily smooth) test configuration, and let $\co\in \scA$. Using the $\partial\bar\partial$-lemma, define a family of smooth functions $\varphi(t)$ by $$\alpha(t)^*\co - \co = i\ddbar \varphi_t.$$ We define the \emph{hamiltonian} for $\alpha(t)$ to be $h_{\alpha} = \alpha(t)_*\dot\varphi(t)$.\end{definition}

The function $h_{\alpha}$ is then a hamiltonian for the corresponding $S^1$-action in the usual sense on the smooth locus of $\X$ (note that the smooth locus admits a $\C^*$-action as the action preserves the dimension of the tangent space). In particular it is indeed independent of $t$ on this locus, and by smoothness is independent of $t$ everywhere as claimed. 

Using this definition of the hamiltonian, we can mimic the analytic definition of the norms and inner products given in the projective case. Denote by $V=\int_X [\omega]^n$ the volume of $(X,[\omega])$. Mirroring the projective case, set $\hat h_{\alpha} =\frac{1}{V}\int_{\X_0}h_{\alpha}\co_0^n$.

\begin{definition}\label{kahlerip}We define the \emph{inner product} $\langle \alpha, \beta \rangle$ by $$\langle \alpha, \beta \rangle =\int_{\X_0}(h_{\alpha} - \hat h_{\alpha})(h_{\beta} - \hat h_{\beta})\co_0^n,$$ where $h_{\alpha}, h_{\beta}$ are hamiltonians for $\alpha,\beta$ respectively as defined above. We analogously define, for example, the inner product of $\alpha$ or $\beta$ with itself. Note also these definitions make sense for general elements of $\Lie (T)$ not necessarily generating a $\C^*$-action. \end{definition}

 This is just the Futaki-Mabuchi inner product when $\X_0$ is smooth \cite{FM}. We can similarly define the $L^p$-norm.

\begin{definition} We define the $L^p$\emph{-norm} $\|(\X,\L)\|_p$ of $(\X,\L)$ by $$\|(\X,\L)\|^p_p = \int_{\X_0} \left|h_{\alpha} -\hat h_{\alpha}\right|^p\co_0^n.$$ \end{definition}

It is natural to ask if the $L^p$-norm can be formulated as an intersection number, in a similar way to the Donaldson-Futaki invariant. This is the case when $\X$ is projective and $p$ is an even integer, and seems very likely in the K\"ahler case. We discuss this further in Remark \ref{intersections-analytic}.

\begin{remark} When $X$ is projective and $[\omega]=c_1(L)$, Sz\'ekelyhidi defines the $L^2$-norm of arbitrary filtrations of the co-ordinate ring of $(X,L)$, which play the role of generalised test configurations \cite[Equation (5)]{GS2}. One example of such an object is a test configuration with an irrational line bundle. One can show that the $L^2$-norm we have defined here equals Sz\'ekelyhidi's $L^2$-norm of a filtration in this case, by approximating both objects with genuine projective test configurations. \end{remark}

We will later prove the following.

\begin{proposition} The norm and inner product depend only on the class $\scA$, and not on the choice of $\co \in \scA$.\end{proposition}

As in the projective case, for $T\subset \Aut(X,[\omega])$ a torus of automorphisms, let us say a K\"ahler test configuration $(\scX,\scA)$ is $T$\emph{-invariant} if it admits a vertical torus action which commutes with $\alpha$ and restricts to the usual action of $T$ on $(\X_t,\scA_t)$ for all $t\neq 0$. Picking an orthogonal basis $\beta_1,\hdots, \beta_d$ of $\C^*$-actions generating $T$, we denote \[\DF_T(\X,\scA) = \DF(\X,\scA) - \sum_{i=1}^d \frac{\langle\alpha,\beta_i\rangle}{\langle \beta_i, \beta_i \rangle}F(\beta).\] Here $F(\beta)$ is the usual (K\"ahler) Futaki invariant associated to $\beta$ on the general fibre (which can also be defined for general elements of $\Lie (T)$). Our definition of relative K-stability is now just as in the projective case.

\begin{definition} We say that $(X,[\omega])$ is \emph{relatively K-stable} if $\DF_T(\X,\scA)>0$ for all test configurations whose projection has positive norm $\|(\X,\scA)\|_2>0$, we have  We sometimes call this \emph{K\"ahler relative K-stability} for clarity. \end{definition}

\begin{remark} We expect that the technical condition on the positivity of the norm can be removed using a forthcoming result of Sj\"ostrom Dyrefelt \cite{ZSD-thesis}. \end{remark}

If one also has $F(\beta_i)=0$ for all $i$, we simply say that $(X,[\omega])$ is \emph{(K\"ahler) equivariantly K-polystable}. This is the notion relevant to the existence of constant scalar curvature K\"ahler metrics on K\"ahler manifolds.

By Example \ref{ACGTF-example}, K\"ahler relative K-stability is a stronger notion than projective relative K-stability when $[\omega]$ is the first Chern class of an ample line bundle (so $X$ is projective).

\begin{remark}\label{YTD-remark} Sz\'ekelyhidi's analogue of the Yau-Tian-Donaldson conjecture states that a smooth polarised variety $(X,L)$ admits an extremal metric if and only if $(X,L)$ is relatively K-stable \cite{GS}. As discussed by Sz\'ekelyhidi \cite{GS}, one very likely needs to strengthen the definition of (projective) relative K-stability for this to be true. In light of Example \ref{ACGTF-example}, it is natural to ask if our notion of K\"ahler relative K-stability is actually the correct strengthening. 

Although K\"ahler relative K-stability is almost certainly the \emph{weakest} plausible notion to imply the existence of an extremal K\"ahler metric, it may be too optimistic to hope the converse is true. In the better understood projective setting of constant scalar curvature K\"ahler metrics with $\Aut(X,L)$ discrete, it is commonly believed that one needs to strengthen the definition of K-stability to either filtration K-stability \cite{GS2} or the even stronger notion of uniform K-stability \cite{RD,BHJ1}. These notions both rule out the phenomenon explained in Example \ref{ACGTF-example}. Thus it may be that one needs a stronger version of K\"ahler relative K-stability to imply the existence of an extremal K\"ahler metric. 
\end{remark}

\section{Lower bounds on the Calabi functional}
\subsection{Preliminaries on geodesics and the Mabuchi functional}\label{prelim-mabuchi-sec}

By analogy with Donaldson's work on the Hitchin-Kobayashi correspondence, Mabuchi introduced a functional on the space of K\"ahler metrics in a fixed K\"ahler class which conjecturally ``detects'' the existence of a cscK metric in that class \cite{TM}. The properties of this functional will be key to proving the lower bound on the Calabi functional.

Let $(X,\omega)$ be an $n$-dimensional K\"ahler manifold. Denote by $$\scH_{\omega} = \{\varphi\in C^{\infty}(X,\R): \omega_{\varphi}=\omega+i\ddbar\varphi > 0\}$$ the space of K\"ahler potentials in the K\"ahler class $[\omega]$.

\begin{definition}\cite{TM}\label{Mabuchi1} For $\varphi\in \scH_{\omega}$, let $\{\varphi_t: t\in [0,1]\}$ be a smooth path of K\"ahler potentials with $\varphi_0=0, \varphi_1 = \varphi$. We define the \emph{Mabuchi functional} $\scM_{\omega}:  \scH_{\omega} \to \R$ to be $$\scM_{\omega}(\varphi)=-\int_0^1 \int_X \dot{\varphi}_t(S(\omega_t) - n\mu(X,[\omega]))\omega_t^n\wedge dt,$$ where $S(\omega_t)$ is the scalar curvature. \end{definition}

Part of the definition is the statement that the Mabuchi functional is independent of path chosen. The Mabuchi functional also admits a more explicit formula as follows.

\begin{proposition}\label{chen-tian-formula}\cite{XC2,GT2} The Mabuchi functional can be written as a sum $\scM_{\omega}(\varphi) = H_{\omega}(\varphi)+E_{\omega}(\varphi),$ where \begin{align*} H_{\omega}(\varphi) &=\int_X \log\left(\frac{\omega_{\varphi}^n}{\omega}\right)\omega_{\varphi}^n, \\ E_{\omega}(\varphi) &= \mu(X,[\omega])\frac{n}{n+1} \sum_{i=0}^n\int_X \varphi \omega^i \wedge \omega^{n-i}_{\varphi}  - \sum_{i=0}^{n-1}  \int_X \varphi \Ric \omega \wedge \omega^i \wedge \omega_{\varphi}^{n-1-i}.\end{align*} \end{proposition}

One should compare this to the intersection-theoretic formulation of the Donaldson Futaki invariant given in Proposition \ref{DF-compact-proj}. 

Mabuchi also defined a Riemannian metric on the space $\scH_{\omega}$, which gives a notion of geodesics in $\scH_{\omega}$. 

\begin{definition} We say $\varphi_t\in \scH_{\omega}$ is a \emph{geodesic} if $$ \ddot\varphi_t -\frac{1}{2}|\nabla \dot\varphi |_t^2 = 0,$$ where the norm and Riemannian gradient are taken with respect to the Riemannian metric induced by $\omega_{\varphi_t}$.\end{definition}

As the space $\scH_{\omega}$ is infinite dimensional, smooth geodesics do not necessarily exist, as one is solving a PDE rather than an ODE. However, provided one interprets the equation appropriately, weak solutions often exist. 

Note that a path of K\"ahler potentials $\varphi_t$ as above is a smooth function on $X\times [0,1]$. It therefore extends to a smooth function $\Phi$ on the manifold $X\times\Delta$, where $\Delta\subset\C$ is the (closed) unit disc, by assuming the extension is radially symmetric on $\Delta$. Now let $\pi_1: X\times \Delta\to X$ be the projection onto the first factor, and set $$\Omega_{\Phi} = \pi_1^*\omega + i\ddbar \Phi.$$

\begin{proposition}\label{semmes-donaldson}\cite{SS1,SS2,SD3} The path $\varphi_t$ is a geodesic if and only if \begin{equation}\label{volume-geodesic}\Omega_{\Phi}^{n+1}=0\end{equation} on $X\times\Delta$.\end{proposition}

The geodesic then becomes a degenerate Monge-Amp\`ere equation, and this reformulation immediately furnishes a notion of a weak geodesic.

\begin{definition} We say that a path $\varphi_t$ of plurisubharmonic functions is a \emph{weak geodesic} if it satisfies equation (\ref{volume-geodesic}) in the sense of pluripotential theory. \end{definition}

The following result of Chen proves an important regularity property of solutions of the geodesic equation.

\begin{theorem}\cite{XC} Weak geodesics are automatically $C^{1,\bar{1}}$ regular, i.e. $\ddbar \varphi_t \in L^{\infty}_{loc}$. \end{theorem}

$C^{1,\bar{1}}$ regularity is weaker than being $C^{1,1}$ in the usual H\"older sense, but implies $\varphi_t$ is in the H\"older space $C^{1,\alpha}$ for all $\alpha<1$. The above regularity result cannot, in general, be improved \cite{LV}. 

Note that the explicit formulation of the Mabuchi functional given in Proposition \ref{chen-tian-formula} implies that the Mabuchi function extends in a natural way to $C^{1,\bar{1}}$ potentials. The key property of the Mabuchi functional that we will need is the following deep result of Berman-Berndtsson. 

\begin{theorem}\cite{BB} The Mabuchi functional is continuous and convex along $C^{1,\bar{1}}$ geodesics.\end{theorem}  This is an observation of Mabuchi when the geodesic is smooth \cite{TM}. 

Given a K\"ahler test configuration for a K\"ahler manifold $(X,[\omega])$, one can canonically associate a $C^{1,\bar{1}}$-geodesic emanating from $\omega$, in a similar manner to Proposition \ref{semmes-donaldson} \cite{PS,RB,RB2,ZSD}. Let $\X_{\Delta}$ be the pre-image of $\Delta\subset\C$ in $\X$. 

\begin{proposition}\label{geodesic-regularity} Let $(\X,\scA)$ be a test configuration. The equation $$\co^{n+1}=0$$ admits a unique $S^1$-invariant solution on $\X_{\Delta}$ such that $\co|_{\partial\Delta} = \omega$ and $\co\in \scA$. Moreover, $\co$ is $C^{1,\bar{1}}$ on $\X\backslash \{\X_0\}$. \end{proposition}

Taking also a smooth relatively K\"ahler metric $\eta\in\scA$, we therefore have two paths in the space of (possibly singular) K\"ahler metrics in $[\omega]$ obtained by setting \begin{align*} \Omega_1 - \alpha(t)^*\Omega_t &= i\ddbar \varphi_t, \\ \eta_1 - \alpha(t)^*\eta_t &= i\ddbar \psi_t.\end{align*} It will also be useful to set $\omega_t = \alpha(t)^*\co_t$.  The path $\psi_t$ is often called a ``subgeodesic'' in the literature. It is also convenient to set $s = -\log |t|^2$, so that $s\to\infty$ corresponds to $t\to 0$. 

\begin{theorem}\cite{ZSD,DR,BHJ2}\label{Mabuchi-expansion} Suppose $\X$ is smooth. Then the Mabuchi functional satisfies $$\lim_{s\to\infty} \frac{\scM(\varphi_{s})}{s} = M^{NA}(\X,\scA).$$ The same expansion holds using instead $\psi_s$. Here $M^{NA}(\X,\scA)$ is the non-Archimedean Mabuchi functional of $(\X,\scA)$ \cite{ZSD,BHJ1}, which satisfies $M^{NA}(\X,\scA)\leq \DF(\X,\scA)$ with equality if and only if $\X_0$ is reduced. \end{theorem}

The version of the above result using the geodesic is due to Sj\"ostr\"om Dyrefelt \cite{ZSD}. We expect the above expansion should hold along the geodesic even in the case that $\X$ is just normal; this is not even known in the projective case. This expansion \emph{does} hold for normal $\X$ along the smooth path \cite{DR,ZSD}, though we will not use this. Results of this form go back to the seminal work of Tian in the setting of K\"ahler-Einstein metrics \cite{GT}; we refer to \cite{DR} for a more extensive bibliography.

We will also use the concept of the norm of a geodesic. 

\begin{definition} We define the \emph{norm} of a geodesic $\varphi_t$ to be $$\|\varphi_t\|^p_p = \int_X |\dot\varphi_t - \hat\varphi_t|^p\omega_t^n,$$ where $\hat \varphi_t = \frac{1}{V}\int_{X}\dot \varphi_t \omega_t^n$.\end{definition}

The following justifies that this is indeed the norm of the geodesic, rather than just the potential.

\begin{lemma}\label{Berndtsson-norm}\cite[Lemma 2.1]{BoB} The value $\|\varphi_t\|_p$ is independent of $t$. Hence if $\|\varphi_t\|_p=0$ for one (equivalently any) $p$, then $$\varphi_t=t\varphi+\varphi_0,$$ where $\varphi,\varphi_0\in C^{\infty}(X,\R)$. \end{lemma}

Note that the second part of the lemma follows from the first by taking $p=2$. One consequence of the above is that $\int_{X}\dot \varphi_t \omega_t$ is constant along the geodesic. We then set $$\hat \varphi = \frac{1}{V}\int_{X}\dot \varphi_t \omega_t^n.$$

\subsection{The Calabi functional}\label{calabi-sec}

Denote by $\mathbb{T}$ the space of test configurations for $(X,[\omega])$. The main result of this section is the following:

\begin{theorem}\label{lowerboundcalabi} We have \[ \inf_{\omega\in [\omega]} \|S(\omega)-\hat S\|_p \geq -\sup_{(\X,\co)\in \mathbb{T}} \frac{\DF(\X,\scA)}{\|(\X,\scA)\|_q}.\] \end{theorem}

We will first obtain an analogous result involving geodesics and their norms using the Mabuchi functional, which is very similar to a result of Berman-Berndtsson \cite[Corollary 1.2]{BB} (and essentially equivalent when $p=q=2$), by using their deep convexity results. For a given test configuration, we will then use its associated geodesic along with a perturbation argument to obtain the above result. 

\begin{proposition} Let $\varphi_s$ be a geodesic. Then for every H\"older conjugate pair $(p,q)$ we have \[ \inf_{\omega\in [\omega]} \|S(\omega)-\hat S\|_p \geq \frac{\lim_{s\to\infty }s^{-1}\scM(\varphi_s)}{\|\varphi_s\|_q}. \] \end{proposition}

\begin{proof} 

Since the geodesic $\varphi_t$ is $C^1$, its derivative is continuous. H\"older's inequality gives \begin{equation}\label{holder}\left(\|S(\omega)-\hat S\|_p\right)\left(\|\dot\varphi_0-\hat\varphi\|_q\right) \geq \int_X (\dot\varphi_0-\hat\varphi)(S(\omega)-\hat S)\omega^n.\end{equation} Note that, as used in \cite[Equation (4.17)]{TH}, $$\int_X (\dot\varphi_0-\hat\varphi)(S(\omega)-\hat S)\omega^n = \int_X \dot\varphi_0(S(\omega)-\hat S)\omega^n.$$ Remark that $\|\varphi_0\|_q = \|\varphi_s\|_q$ for all $s$ by Lemma \ref{Berndtsson-norm}.  

The Mabuchi functional is continuous and convex along the geodesic $\varphi_s$ \cite{BB}. Then by elementary properties of convex functions and \cite[Lemma 3.5 and proof of Corollary 1.2]{BB} we have  \begin{align*}\left(\|S(\omega)-\hat S\|_p\right)\left(\|\dot\varphi_0-\hat\varphi\|_q\right) &\geq \int_X \dot\varphi_0(S(\omega)-\hat S)\omega^n, \\ & \geq  - \lim_{s\to\infty}\frac{\scM(\varphi_s)}{s}. \end{align*} As this is true for each $\omega$, we obtain the result.

\end{proof}

The use of H\"older's inequality in the above proof is essentially motivated from the fact that the term $S(\omega)-\hat S$ arises as a moment map, and seems to have first been used by Donaldson \cite{SD}. The above proof should be thought of as an analogue of part of the Kempf-Ness Theorem, which gives lower bounds on the norm squared of moment maps, using the Cauchy-Schwarz inequality. Compare, for example, \cite{SD,XC3,TH} and the survey \cite{GRS}.  

The next step is the following result relating the various norms.

\begin{theorem}\label{geodesic-test-config-norm}  The $L^p$-norm of $(\X,\scA)$ is independent of choice of $\Omega\in\scA$. If $\X$ is smooth, then the $L^p$-norm of a test configuration equals the $L^p$-norm of the associated geodesic. \end{theorem}

\begin{proof} The starting point of the proof is the following result, which the author learned from G. Sz\'ekelyhidi: let $(X,[\omega])$ be a compact K\"ahler manifold with a hamiltonian $S^1$-action. Setting $h_{\omega}$ to be the hamiltonian for any $\omega \in [\omega]$, then for any $C^1$-function $f$ on $X$, the integral $\int_X f(h_{\omega})\omega^n$ is independent of $\omega\in [\omega]$. The proof follows by a simple calculation (fixing any two K\"ahler metrics and differentiating the integral along the line of K\"ahler metrics between them), and generalises for example the classical invariance of the Futaki invariant and $L^2$-norm of Futaki-Mabuchi \cite{FM}.

It follows that $\int_X |h_{\omega} - \hat h_{\omega}|^p\omega^n$ is independent of $\omega\in [\omega]$ for $p$ an even integer, and also for general $p$ by approximating the $p$-norm by a sequence of $C^1$-functions $f$. Moreover, if $\omega$ and $[\omega]$ are just \emph{semi}-positive, but $h_{\alpha}$ satisfies the hamiltonian equation in the sense of Definition \ref{hamiltonian-def}, it follows from continuity of the integral that the integral is still independent of $\omega$ by approximating $\omega$ and $[\omega]$ by K\"ahler classes (where by above the integral is independent of the approximation once the class is fixed). The results above hold for weak $C^{1,\bar{1}}$-K\"ahler metrics (or the semi-positive analogue) by another approximation argument (with our usual definition of the hamiltonian), this time approximating the weak (semi)-K\"ahler metric by smooth (semi)-positive forms in the same class. %

Assuming $\X$ is smooth, Proposition \ref{geodesic-regularity} implies that the weak K\"ahler metric $\eta$ induced by the geodesic is globally $C^{1,\bar{1}}$, in the sense that the potentials are $C^{1,\bar{1}}$ with respect to some smooth reference metric. Thus the hamiltonian $h_{\eta}$ is a continuous function. Pick some smooth $\co\in[\co]$ with hamiltonian $h_{\Omega}$, with which we will calculate the $L^p$-norm of $(\X,\scA)$. We claim that \begin{equation}\label{pf:hamiltoniansequal}\int_{\X_0}|h_{\co} - \bar h|^p \co_0^n = \int_{\X_0}|h_{\eta} - \bar h|^p \eta_0^n.\end{equation}Indeed, this follows by passing to a resolution of singularities $p: \Y\to\X$ such that $\Y_0$ is a simple normal crossings divisor, which simply means that as a cycle $\Y_0 = \sum_i a_i \Y_{0,i}$ with $\Y_{0,i}$ smooth (compact) K\"ahler manifolds and $a_i\in\N$  (recall that log resolution of singularities holds in the K\"ahler category by \cite[Lemma 2.2]{CMM}). Thus equation (\ref{pf:hamiltoniansequal}) follows from the discussion above applied to each $\Y_{0,i}$ with respect to the class $(p^*[\co])|_{\Y_{0,i}}$. This argument also shows that the $L^p$-norm is independent of choice of smooth $\Omega\in\scA$, even when $\X$ is singular. 

We now use the invariance of the norm along the geodesic, namely Lemma \ref{Berndtsson-norm}: the value $$\int_{X = \X_1} |\dot\varphi_s - \hat\varphi_s|^p\omega_s^n = \int_{\X_t} |h_{\eta} - \hat h|^p\eta_t^n$$ is independent of $s$, hence equals the corresponding integral over $\X_0$, which in turn equals the norm of the test configuration by equation (\ref{pf:hamiltoniansequal}).

\end{proof} 

\begin{remark}There is another, slightly less appealing, proof of the above result using the fact that, for some smooth $\Omega\in\scA$, we have $[\Omega]=[\eta]$. As $\X$ is smooth this implies that the corresponding families of potentials $\varphi_s$ and $\psi$ are uniformly bounded in $C^1$ as $s\to\infty$. This can then be used to show directly that $$\lim_{s\to\infty} \int_X |\dot \varphi_s - \hat \varphi|^p\eta_s^n = \lim_{s\to\infty} \int_X |\dot \psi_s - \hat \psi|^p\Omega_s^n,$$ which is all that is required.  \end{remark}

Theorem \ref{geodesic-test-config-norm} is due to Hisamoto in the projective case, who uses very different techniques \cite[Theorem 1.2]{TH} which do not straightforwardly extend to the K\"ahler setting. Essentially the same argument gives the following: 

\begin{corollary} The inner product of a test configuration $(\X,\scA)$ with a vector field is independent of $\eta\in\scA$.\end{corollary}

Relating this to algebraic geometry, we have have proven:

\begin{corollary} Denote by $\mathbb{T}_{sm}$ the set of test configurations for $(X,[\omega])$ with $\X$ smooth and with reduced central fibre. We have $$\inf_{\omega\in [\omega]} \|S(\omega)-\hat S\|_p \geq -\sup_{(\X,\co)\in \mathbb{T}_{sm}} \frac{\DF(\X,\scA)}{\|(\X,\scA)\|_q}.$$\end{corollary}

\begin{proof} This follows immediately by the above, using Theorem \ref{Mabuchi-expansion}. Indeed, under the hypotheses the Donaldson-Futaki invariant equals the non-Archimedean Mabuchi functional. \end{proof}

The difference between this and the general form of the lower bound on the Calabi functional is the requirement that the total space be smooth. However, as we are only proving an inequality we are able to perturb to this case as follows.

\begin{proposition}\label{perturbation} Given an arbitrary test configuration with $(\X,\scA)$ and $\epsilon>0$, there exists a smooth test configuration $(\Y,\scB)$ with $$ \left|\frac{\DF(\X,\scA)}{\|(\X,\scA)\|_q} -  \frac{\DF(\Y,\scB)}{\|(\Y,\scB)\|_q}\right| < \epsilon.$$ Moreover, given a smooth test configuration $(\X,\scA)$ with non-reduced central fibre, there exists a test configuration $(\Y,\scB)$ with reduced central fibre and an $\epsilon\geq 0$ satisfying $$\frac{\DF(\Y,\scB)}{\|(\Y,\scB)\|_q}- \epsilon\leq \frac{\DF(\X,\scA)}{\|(\X,\scA)\|_q}.$$ \end{proposition}

\begin{proof} We first consider the smoothness statement. The proof constructs an explicit pair $(\Y,\scB)$ for each $\epsilon>0$. We let $f: \Y\to\X$ be a resolution of singularities, and set $\scB = f^*\scA - \delta [E]$ where $E$ is the exceptional divisor of the resolution. It is convenient to take a smooth $S^1$-invariant representatives $\co\in \scA$ and $\nu$ of $[E]$.

By \cite{DR}, we have $\DF(\Y, f^*[\co] - \delta [\nu]) = \DF(\X, [\co]) + O(\delta)$. A similar conclusion for the norms will be enough to conclude the smoothness result. 

Denote the hamiltonians on $\X$ and $\Y$ respectively by $h_{\X}, h_{\Y,\delta}$. Then one sees by definition of the hamiltonians that $\hat h_{\Y,\delta}  = f^*h_{\X} + \delta h_E$ for some smooth function $h_E$. The smoothness result follows, since the $L^q$-norm of $(\Y, f^*[\co] - \delta [\nu])$ is defined as $$\|(\Y,\scA-\delta [E])\|^q_q = \int_{\Y_0}| h_{\Y,\delta} - \hat h_{\Y,\delta}|^q(f^*\co_0 - \delta \nu_0)^n.$$

Thus it suffices to show that one can also assume the central fibre is reduced. Take $f: \Y\to\X$ to be a semi-stable reduction of $\X$. Thus, $\Y\to\X$ is the normalisation of the base change induced from a finite cover of $\C\to\C$ of the form $t\to t^d$ where $t$ is the co-ordinate on $\C$. For $d$ sufficiently divisible, $\Y_0$ will be reduced \cite[proof of Proposition 7.15]{BHJ1}. Pick a smooth form $\co\in\scA$. Then at the level of potentials, setting $\alpha(t)^*\co-\co = i\ddbar\psi(t)$, this base change reparametrises $\psi(t)$ to $\psi(dt)$ \cite[p1013]{RB}. By definition of the hamiltonian, and using the obvious notation, this means $h_{\Y} = dh_{\X}$. We therefore have $$(d\|(\X,\scA)\|_p)^p = \int_{\X_0}|dh_{\X} - d\hat h_{\X}|^p = \int_{\Y_0}|h_{\Y} - f^*\hat h_{\Y}|^p.$$ By taking a resolution of singularities of $\scY$ and perturbing the class $f^*\scA$ to a relatively K\"ahler class on the resolution, we obtain a test configuration with reduced central fibre and the same norm up to a term of order $\epsilon$. Similarly we have $\DF(\Y,f^*\scA) \leq d\DF(\X,\scA)$ \cite[Proposition 7.14]{BHJ1}\cite[Proposition 3.15]{ZSD}\cite[Remark 2.21]{DR}, with equality if and only if $\X_0$ is already reduced. The result follows.

 \end{proof}
 
This gives a new proof of a result of Arezzo-Della Vedova-La Nave \cite[Theorem 1.5]{ADV}, who proved Proposition \ref{perturbation} by a delicate analysis of weight polynomials when $X$ is projective and $p=2$.

Theorem \ref{lowerboundcalabi} is an immediate corollary.

\subsection{Norms}\label{norms-sec}

We are now in a position to prove Theorem \ref{intro-norms}. 

\begin{theorem}\label{norms}Let $(\scX,\scA)$ be a test configuration. The following are equivalent.
\begin{itemize}
\item[(i)] The $L^p$-norm $\|(\scX,\scA)\|_p$ vanishes for all $p$.
\item[(ii)] The $L^p$-norm $\|(\scX,\scA)\|_p$ vanishes for some $p$.
\item[(iii)] The minimum norm $\|(\scX,\scA)\|_m$ vanishes. \end{itemize} Suppose that $\X$ is smooth. Then these are equivalent to: \begin{itemize} \item[(iv)] The geodesic associated to $(\scX,\scA)$ is trivial. \end{itemize}
\end{theorem}

In the projective case, the equivalence of $(i)$ and $(ii)$ is obvious from Donaldson's definition \cite[Section 5.1]{SD}. $(i) \Leftrightarrow (iii)$ was proven independently by the author \cite[Theorem 1.3]{RD} and Boucksom-Hisamoto-Jonsson \cite[Theorem A]{BHJ1}, and $(i) \Leftrightarrow (iv)$ is due to Hisamoto \cite[Theorem 1.2]{TH}. Our logical equivalences will be proven in a slightly different order in the K\"ahler case, and the proofs are very different.

First we recall the definition of the minimum norm. Let $(\X,\scA)$ be a test configuration for $(X,[\omega])$. Denote by $$f: X\times\pr^1 \dashrightarrow \X$$ the natural bimeromorphic map, and take a resolution of indeterminacy: \[
\begin{tikzcd}
\Y \arrow[swap]{d}{q} \arrow{dr}{g} &  \\
X\times\pr^1 \arrow[dotted]{r}{f} & \X
\end{tikzcd}
\]

\begin{definition}\cite{RD,BHJ1} We define the \emph{minimum norm} of $(\X,\scA)$ to be $$\|(\X,\scA)\|_m = (g^*\scA).[q^*\omega]^n - \frac{(g^*\scA)^{n+1}}{n+1}.$$ The minimum norm is also called the \emph{non-Archimedean $J$-functional}, and is independent of choice of resolution of indeterminacy.
\end{definition} The definition we use here differs slightly from the version used in \cite{RD}, but is equivalent.

Just as with the Mabuchi functional and the Donaldson-Futaki invariant, the minimum norm occurs as the slope of a functional on the space of K\"ahler metrics.  

\begin{definition}\cite[p46]{GT2} Let $\varphi$ be a K\"ahler potential for $\omega$. We define the \emph{J-functional} of $\varphi$ as $$J(\varphi) := \int_X \varphi \omega^n - \frac{1}{n+1}\left(\sum_{i=0}^n \int_X \varphi \omega^i\wedge\omega_{\varphi}^{n-i}\right).$$\end{definition}

\begin{theorem}\cite{DR,BHJ2,ZSD}\label{jthm} Suppose $\X$ is smooth. Then $$\lim_{s\to\infty}\frac{dJ(\varphi_{s})}{ds} = \|(\X,\scA)\|_m,$$ where $\varphi_s$ is the geodesic associated to $(\X,\scA)$ as in Theorem \ref{Mabuchi-expansion}.\end{theorem}

The precise version we need along the geodesic follows from \cite{ZSD}, using that the J-functional is continuously differentiable along the geodesic.

\begin{proof}[Proof of Theorem \ref{norms}]  \ 

$(i) \Leftrightarrow (ii)$ This is obvious by definition of the norm.

$(ii) \Leftrightarrow (iii)$ We first assume that $\X$ is smooth. Let $\varphi_s$ be the geodesic associated to the test configuration.  By the derivative  of \cite[Proposition 5.5]{DarRub}, there exists a $C>1$ independent of $\varphi_s$ such that $$\frac{1}{C}\frac{dJ(\varphi_{s})}{ds}  \leq \|\varphi_s\|_1 \leq C \frac{dJ(\varphi_{s})}{ds}.$$ Indeed, the first inequality is proven to follow from \cite[Proposition 5.5]{DarRub} in \cite[Proposition 5.26]{DR}. The second inequality then follows from the characterisation of the $d_1$-pseudometric \cite[Definition 4.2]{DarRub} given in \cite[Theorem 4.3]{DarRub}. Taking the limit as $s\to\infty$ gives the result by Theorem \ref{jthm} since $\X$ is smooth. But it follows that for all smooth $(\X,\scA)$ we have $$\frac{1}{C}\|(\X,\scA)\|_m\leq \|(\X,\scA)\|_1 \leq C\|(\X,\scA)\|_m$$ for some $C>1$ \emph{independent of} $(\scX,\scA)$. We thus obtain the same inequality for \emph{all} $(\X,\scA)$ by a perturbation argument similar to Proposition \ref{perturbation}. This then provides the desired conclusion.

$(ii) \Leftrightarrow (iv)$ This is the hardest part of the proof in the projective case, and for us follows from Theorem \ref{geodesic-test-config-norm}.
\end{proof}

This also shows that ``$L^1$-uniform K-stability'' in the sense of Sz\'ekelyhidi \cite{GS4} is equivalent to uniform K-stability with respect to the minimum norm \cite{RD,BHJ1} (also called J-uniform K-stability), i.e. we have proven that there exists a universal constant $C>1$ such that $$\frac{1}{C}\|(\X,\scA)\|_m \leq \|(\X,\scA)\|_1\leq C\|(\X,\scA)\|_m.$$ In the projective case this is due to Boucksom-Hisamoto-Jonsson \cite{BHJ1}. Again in the projective case, we proved \cite[Theorem 1.3]{RD}\cite[Theorem 6.19]{BHJ1} we proved that a test configuration satisfies $\|(\scX,\scA)\|_m>0$ if and only if $\X$ is not equivariantly isomorphic to $X\times\C$ (our definition of a test configuration requires $\X$ to be normal); it would be interesting to prove this in the K\"ahler case. 

\section{Extremal metrics and relative K-stability}\label{relative-proof-sec}
\subsection{Relative K-semistability}

The main result of this section is:

\begin{theorem}\label{relsemi} Suppose $(X,[\omega])$ admits an extremal K\"ahler metric. Then $(X,[\omega])$ is relatively K-semistable. \end{theorem}

Here relative K-semistability just means that for each test configuration, we have $\DF_T(\scX,\scA)\geq 0$, where $T$ is a maximal torus of automorphisms. More generally, our proof works with $T$ replaced by a torus containing the extremal vector field (whose definition we will shortly recall). Once we prove some properties of the inner product and the norm, the proof of this will follow from the lower bound on the Calabi functional that we have proven. For these properties, we let $(\scX,\scA)$ be a test configuration with $\C^*$-action $\alpha$ and a vertical $\C^*$-action $\beta$. More generally we allow $\beta \in \mathfrak{t} = \Lie(T)$ not necessarily rational (where we are abusing notation in the obvious manner for rational $\beta$, which then generate a $\C^*$-action). To emphasise the dependence on the $\C^*$-action, we write $\DF(\X,\scA,\alpha)$ for the Donaldson-Futaki invariant with respect to $\alpha$. 

\begin{definition}\label{DF-irrational} We define the \emph{Donaldson-Futaki invariant with respect to} $\alpha+\beta$ to be $$\DF(\X,\scA,\alpha+\beta) = \DF(\X,\scA,\alpha) + F(X,[\omega],\beta),$$ where $F(X,[\omega],\beta)=\int_X h_{\beta}(S(\omega) - \hat S) \omega^n$ is the usual Futaki invariant with respect to $\beta\in\mathfrak{t}$.\end{definition}

\begin{proposition}\label{propa}This is well defined, i.e. agrees with the usual definition when $\beta$ is rational and hence generates a $\C^*$-action. \end{proposition}

\begin{proof} 

We prove this by the relationship between the Donaldson-Futaki invariant and the limit derivative of the Mabuchi functional given in Proposition \ref{Mabuchi-expansion}. The definition of the Mabuchi functional involves a ``fixed'' K\"ahler metric $\omega$, and an another K\"ahler metric $\omega_{\varphi} = \omega + i\ddbar \varphi.$ Instead of writing simply $\scM(\varphi)$, for clarity in proving this result we will denote the Mabuchi functional on $X$ as $\scM_X(\omega,\omega_{\phi})$. In this notation, the key property we will use is the cocycle condition $$\scM_X(\omega_1,\omega_2)+\scM_X(\omega_2,\omega_3)+\scM_X(\omega_3,\omega_1)=0.$$ 

Fix a smooth relatively K\"ahler metric $\co\in\scA$. By the cocycle property we have $$\scM_{\X_1}(\co_1,\alpha(t)^*\beta(t)^*\co_t) = \scM_{\X_1}(\co_1,\alpha(t)^*\co_t) + \scM_{\X_1}(\alpha(t)^*\co_t,\alpha(t)^*\beta(t)^*\co_t).$$ Now $$\scM_{\X_1}(\alpha(t)^*\co_t,\alpha(t)^*\beta(t)^*\co_t) = \scM_{\X_t}(\co_t,\beta(t)^*\co_t).$$ Note that $\scM_{\X_t}(\co_t,\beta(t)^*\co_t) = tF(X,[\omega],\beta)$, since more generally \cite{TM} $$\scM_{\X_t}(\co_t,\beta(s)^*\co_t) = sF(X,[\omega],\beta).$$ The result then follows from Proposition \ref{Mabuchi-expansion}.

\end{proof}

We will also require the following property of the inner product (justifying its name), which follows immediately from the definition. Remark that it holds for general $\beta \in \mathfrak{t}$. 

\begin{lemma}\label{propb} The inner product satisfies $$\langle \alpha+\beta, \alpha+\beta\rangle = \langle \alpha, \alpha\rangle+2\langle \alpha, \beta\rangle + \langle\beta,\beta\rangle.$$ \end{lemma}

Finally we will need to be able to compute the norm of a vertical $\C^*$-action on the central fibre on the central fibre.

\begin{proposition}\label{norm-independence} Let $\beta$ be a vertical $\C^*$-action on $(\X,\scA)$, and for clarity denote its inner product on $\X_t$ (thought of as a product test configuration for $t\neq 0$, and using Definition \ref{kahlerip} for $t=0$) by $\langle\beta_t,\beta_t\rangle$. Then $\langle \beta_t,\beta_t\rangle$ is independent of $t$. \end{proposition}

\begin{proof} The fact that it is independent of $t$ for $t\neq 0$ is due to Futaki-Mabuchi \cite{FM}, since it is also a continuous function in $t$ (by its integral representation) it must be independent of $t$ for \emph{all} t.\end{proof}

Our hypothesis in proving Theorem \ref{relsemi} is that $(X,[\omega])$ admits an extremal metric, i.e. one satisfying $$\bar\partial \nabla^{1,0}S(\omega) = 0.$$ Denote by $\chi = \nabla^{1,0}S(\omega)$ the extremal vector field. This vector field satisfies the key property that \begin{equation}\label{extremal-eqn}\|S(\omega)-\hat S\|_2 = \frac{F(\chi)}{\|\chi\|_2} = \|\chi\|_2.\end{equation} By Proposition \ref{norm-independence}, the norm $\|\chi\|_2$ can equivalently be calculated on $X$ or $\X_0$.

With these results proven, we can prove Theorem \ref{relsemi} in an identical way to the projective case \cite[Theorem 7]{StSz}. We recall their proof for the reader's convenience.

\begin{proof}[Proof of Theorem \ref{relsemi}]Let $\omega\in [\omega]$ be the extremal metric, which exists by hypothesis, and fix a test configuration $(\X,\scA)$. The lower bound on the Calabi functional of Theorem \ref{intromaintheorem} for $p=q=2$ gives \[\|S(\omega)-\hat S\|_2 \geq -\frac{\DF(\X,\scA)}{\|(\X,\scA)\|_2}.\] Note this holds also for the Donaldson-Futaki invariant with respect to $\alpha +\beta$ for any $\beta\in\mathfrak{t}$ by the continuity of  Definition \ref{DF-irrational} (which is essentially a consequence of Proposition \ref{propa}) and continuity of the inner products (whose definition extends in an obvious way to general $\beta \in \mathfrak{t}$). Combining this with equation (\ref{extremal-eqn}), we get \begin{equation}\label{lowerboundextremal} \frac{\DF(\X,\scA)}{\|(\X,\scA)\|_2} \geq - \|\chi\|_2. \end{equation}

We can assume that $\langle \alpha, \chi\rangle = 0,$ replacing $\alpha$ with $\alpha+c\chi$ for some $\chi$ if not, since by Proposition \ref{propa} the value $\DF_T(\scX,\scA,\alpha)$ is independent of this replacement (where we are including $\alpha$ for clarity). Here we mean the Donaldson-Futaki invariant in the sense of Definition \ref{DF-irrational}. We therefore wish to show $\DF(\X,\scA,\alpha)\geq 0$, since $\alpha$ is now assumed to be orthogonal to $\chi$. 

Suppose not, and set $\lambda$ to be such that $\DF(\X,\scA,\lambda\alpha) = -\lambda\|\alpha\|_2$. Twist the action on $(\scX,\scA)$ once again so that the action is given by $\lambda\alpha-\chi$. We claim that $$\frac{\DF(\X,\scA,\lambda\alpha-\chi)}{\|\lambda\alpha-\chi\|_2} < - \|\chi\|_2.$$ Indeed, since $\langle \alpha,\chi\rangle = 0$, we have $$ \DF(\X,\scA,\lambda\alpha-\chi) = -\| \lambda\alpha\|_2^2 - \|\chi\|_2^2 = -\| \lambda\alpha -\chi\|_2^2.$$ This gives $$\frac{\DF(\X,\scA,\lambda\alpha-\chi)}{\|\lambda\alpha-\chi\|_2} = - \| \lambda\alpha -\chi\|_2  < -\|\chi\|_2,$$ contradicting equation (\ref{lowerboundextremal}). Hence $\DF(\X,\scA,\alpha)\geq 0$ as required.

\end{proof}

\subsection{The proof of relative K-stability}
We now prove relative K-stability of K\"ahler manifolds admitting extremal metrics.

\begin{theorem}\label{bodymain} If $(X,[\omega])$ admits an extremal metric, then it is relatively K-stable.\end{theorem}

This furnishes the following corollary in the cscK case.

\begin{corollary} If $(X,[\omega])$ admits an cscK metric, then it is equivariantly K-polystable. \end{corollary} 

\begin{proof} This is immediate by the definition of relative K-stability. Indeed, in this case the Futaki invariant $F(\beta)$ vanishes for all vector fields $\beta$, as $[\omega]$ contains a cscK metric. \end{proof}

The proof of Theorem \ref{bodymain} will be by a perturbation argument. Namely, if $(X,[\omega])$ is strictly relatively K-semistable, we will perturb it to an unstable K\"ahler manifold. On the other hand the extremal condition is an ``open'' condition. This is roughly a K\"ahler analogue of the strategy of Stoppa-Sz\'ekelyhidi \cite{StSz}, which in turn is related to Stoppa's strategy in the cscK case \cite{JS}. The key analytic result will be the following, due to Arezzo-Pacard-Singer \cite{APS}. We use the formulation of \cite[Theorem 8]{StSz}.

\begin{theorem}\cite{APS}\label{APS} Suppose $(X,[\omega])$ admits an extremal metric, and let $p\in X$ be a point fixed by a maximal torus of automorphisms. Let $\pi: Bl_pX\to X$ be the blowup and let $E\subset Bl_pX$ be the exceptional divisor. Then $(Bl_pX, \pi^*[\omega]-\epsilon [F])$ admits an extremal metric for all $0<\epsilon\ll 1.$ \end{theorem}

Remark that since $p$ is fixed by the torus action, a maximal torus torus of automorphisms of $(X,[\omega])$ induces a maximal torus of automorphisms of $(Bl_pX, \pi^*[\omega]-\epsilon [F])$.

Now suppose $(X,[\omega])$ is strictly relatively K-semistable. This means that there exists a test configuration $(\scX,\scA)$ with $\DF_T(\scX,\scA) = 0$ but with $\|(\scX,\scA)\|_2>0.$ For $p\in X$, let $C=\overline{\C^*.p}$ be the closure of the orbit of $p$ under the $\C^*$-action. Define the \emph{Chow weight} of $p$ to be $$Ch_p(\scX,\scA) = \frac{\scA^{n+1}}{(n+1)[\omega]^n} - \int_{C}\scA.$$ This agrees with the Chow weight of the specialisation of $p$ in the usual geometric invariant theoretic sense when $\X$ is projective and $\scA$ is the first Chern class of a line bundle.

Suppose that $C$ is smooth. Provided $p$ is invariant under a maximal torus, $C$ will be also. Let $\scY=Bl_C\X \to \X$ be the blowup with exceptional divisor $F$. As $C$ is torus invariant, $(Bl_C\X,\scA-\epsilon [E])$ is also torus invariant. It is therefore an equivariant test configuration for $(Bl_pX, \pi^*[\omega]-\epsilon [F])$ \cite[p19]{DR}.

\begin{proposition}\cite{JS,JS2}\cite[Proposition 5.4]{DR}\label{DFcomparison} Suppose $\X$ and $C$ are smooth. The Donaldson-Futaki invariant of $(\scY,\scA-\epsilon [E])$ satisfies $$\DF(\scY,\scA-\epsilon [E]) = \DF(\scX,\scA) - n(n-1)\epsilon^{n-1}Ch_p(\scX,\scA) + O(\epsilon^n).$$ \end{proposition} 

\begin{remark}\label{szk-rational} Sz\'ekelyhidi proves an analogue of Proposition \ref{DFcomparison} for the Futaki invariant of vector fields that are not necessarily rational. \end{remark}

Here we will prove an analogue of this for the norms and inner products. Abusing notation, we will denote by $\beta_i$ an orthogonal basis of generators of the maximal torus of automorphisms of $(X,[\omega])$ \emph{and} of $(Bl_pX, \pi^*[\omega]-\epsilon [E])$. We will also suppress the obvious pullbacks of classes to $\Y$.

\begin{proposition}\label{norm-ip-pert} Suppose $\X_0$ is a simple normal crossings divisor. The inner products satisfy $$\langle \alpha,\beta_i\rangle_{(\scY,\scA-\epsilon [E])} = \langle\alpha,\beta_i \rangle_{(\scX,\scA)} +  O(\epsilon^n),$$ and similarly for $\langle \beta_i,\beta_i\rangle$. \end{proposition}

The next result guarantees the existence of an equivariant ``destabilising point''. 

\begin{proposition}\label{destabilisingpoint} Suppose $\|(\scX,\scA)\|_2>0$. Then there exists a torus invariant $p\in X$ such that $Ch_p(\scX,\scA)>0$. \end{proposition}

These are all the tools needed to prove Theorem \ref{bodymain}. In reading the proof below, it is enlightening to first understand the proof when $\X$ and $C$ are smooth, which significantly simplifies the argument.

\begin{proof}[Proof of Theorem \ref{bodymain}]  We first consider some generalities on test configurations and blowups. Let $(\scX,\scA)$ be a test configuration, and let $p\in X$ be a torus invariant point. Our goal will be to construct auxiliary test configurations from this data, and to work out their numerical invariants. 

Let $\scS$ be an equivariant resolution of singularities of $\X$ such that the proper transform $\hat C$ of $C$ in $\scS$ is smooth and $\X_0$ is a simple normal crossings divisor. By definition $\DF(\scX,\scA) = \DF(\scY,\scA)$. Let $\scB\to\scS$ be the blow-up of $\hat C$. Denote by $E_1$ the exceptional divisor of $\scS\to\scX$, and $E_2$ the exceptional divisor of $\scB\to\scP$. We have the diagram: \[
\begin{tikzcd}
\B \arrow[swap]{d}{} \arrow{dr}{} &  \\
\scS \arrow{r}{} & \X
\end{tikzcd}
\]

The class $\scA-\epsilon^{2n} [E_1]$ is relatively K\"ahler on $\scS$, hence $(\scS,\scA-\epsilon^{2n} [E_1])$ is a test configuration for $(X,[\omega])$, as $E_1$ is supported on the central fibre. Similarly $(\scB,\scA-\epsilon [E_2]-\epsilon^{2n} [E_1])$ is a test configuration for $(Bl_pX, \pi^*[\omega]-\epsilon [E])$, and an analogous formula holds for the various inner products.

We now compare the various numerical invariants. First of all we have \begin{equation}\label{DFonResolutionPert}\DF(\scS,\scA-\epsilon^{2n} [E_1]) = \DF(\scX,\scA)+O(\epsilon^{2n}).\end{equation} Similarly we have \begin{equation}\label{perturb-the-norm} \langle \alpha,\beta_i\rangle_{(\scS,\scA-\epsilon^{2n} [E_1]} =  \langle \alpha,\beta_i\rangle_{(\scX,\scA)} + O(\epsilon^{2n}),\end{equation} and the same holds for the other inner products and also the Chow weight.

Now Proposition \ref{DFcomparison}, together with equation (\ref{DFonResolutionPert}) together with the fact that the Chow weight can be computed on $\scS$ \cite[Lemma 5.29]{DR}, implies that $$\DF(\scB,\scA-\epsilon [E_2]-\epsilon^{2n} [E_1]) = \DF(\scX,\scA) - n(n-1)\epsilon^{n-1}Ch_p(\scX,\scA) + O(\epsilon^n).$$ But then Proposition \ref{norm-ip-pert} together with equation (\ref{perturb-the-norm}) give that $$ \langle \alpha,\beta_i\rangle_{(\scB,\scA-\epsilon [E_2]-\epsilon^{2n} [E_1])} = \langle \alpha,\beta_i\rangle_{(\X,\scA)} + O(\epsilon^n),$$ and the corresponding results hold for the  other inner products.

Now we can complete the proof using the above. Suppose $(X,[\omega])$ is strictly relatively K-semistable. This means there exists a test configuration $(\scX,\scA)$ with $\DF_T(\scX,\scA) = 0$ but with $\|(\scX,\scA)\|_m>0.$ Using Proposition \ref{destabilisingpoint}, let $p\in X$ be a torus invariant point with $Ch_p(\scX,\scA)>0$. Then $(\scB,\scA-\epsilon [E_2]-\epsilon^{2n} [E_1])$ is a test configuration for $(Bl_pX, \pi^*[\omega]-\epsilon [E])$. By the above results, as well as Remark \ref{szk-rational}, we now have the key equation $$\DF_T(\scB,\scA-\epsilon [E_2]-\epsilon^{2n} [E_1])<0.$$ Hence if $(X,[\omega])$ is strictly relatively K-semistable, there exists a torus invariant point $p\in X$ such that $(Bl_pX, \pi^*[\omega]-\epsilon [E])$ is relatively K-unstable.

By assumption $(X,[\omega])$ admits an extremal metric. Moreover, for all torus invariant $p\in X$, by Theorem \ref{APS} the blowup $(Bl_pX, \pi^*[\omega]-\epsilon [E])$ admits an extremal metric. Hence for each torus invariant point $p\in X$, the blowup $(Bl_pX, \pi^*[\omega]-\epsilon [E])$ is also relatively K-semistable. Choosing the point $p$ that makes $(Bl_pX, \pi^*[\omega]-\epsilon [E])$ relatively K-unstable gives a contradiction, hence  $(X,[\omega])$ is relatively K-stable as claimed.

\end{proof}

\begin{remark} When $\X$ is projective, Stoppa-Sz\'ekelyhidi are able to work directly on $\X$ rather than passing to a resolution of singularities, which simplifies the above argument \cite[Theorem 4]{StSz}. As our definition of the Donaldson-Futaki invariant requires working on a smooth total space, this part of our argument is necessarily quite different. Clearly the above proof simplifies somewhat when $\X$ and $C$ are already smooth. \end{remark}
\subsection{Further details}\label{ECW}

We first prove Proposition \ref{destabilisingpoint}. 

\begin{proposition}\label{bodylemmachow} Suppose $\|(\scX,\scA)\|_2>0$. Then there exists a torus invariant $p\in X$ such that $Ch_p(\scX,\scA)>0$. \end{proposition}

\begin{proof} We take for granted \cite[Proposition 5.5]{DR}, which proves the existence of a point $p$ with $Ch_p(\scX,\scA)>0$ provided $\|(\scX,\scA)\|_2>0$, which is not necessarily torus invariant. 

For convenience we normalise the hamiltonian such that $\hat h_{\alpha} = 0$ (this involves changing the K\"ahler class $\scA$ but none of the relevant quantities). It follows from \cite[Lemma 5.17]{DR} that we have $$Ch_p(\scX,\scA) = h_{\alpha}\left(\lim_{t\to 0}\alpha(t).p\right),$$ i.e. the Chow weight is just the value of the hamiltonian at the specialisation of $p$. Pick $p$ such that $Ch_p(\scX,\scA)>0$. Take an orthogonal basis of generators $\beta_i$ of a maximal torus. As $\alpha$ and $\beta_i$ commute, $h_{\alpha}$ is invariant under $\beta_i$. By commutation we also have  $$\lim_{t\to 0} \alpha(t)\left(\lim_{s\to 0} \beta_i(s).p\right) = \lim_{t\to 0} \beta_i(t)\left(\lim_{s\to 0} \alpha(s).p\right).$$ This implies \begin{align*}Ch_p(\scX,\scA) &= h_{\alpha}\left(\lim_{t\to 0}\alpha(t).p\right), \\ & = h_{\alpha}\left(\lim_{t\to 0} \alpha(t)\left(\lim_{s\to 0} \beta_i(s).p\right)\right), \\ &=Ch_{\lim_{s\to 0}\beta_i(s).p}(\scX,\scA).\end{align*} Repeating this for each $i$ gives the result.\end{proof}

One can also give a proof of Proposition \ref{bodylemmachow} using intersection theory, again relying on \cite[Proposition 5.5]{DR} to produce a non-torus invariant destabilising point.  Then the limit $\lim_{t\to 0}\beta_i(t).C$ is a $\beta_i$-invariant curve. Doing this successively we get a curve $\hat C$ which is torus invariant and $\alpha$-invariant. Let $p = \hat C \cap \X_1$, which is a single torus invariant point. Since the action of $\beta$ on $\X$ restricts to the usual action on each fibre $\X_t$, we have that $\overline{\C^*.p}=\hat C$. One can assume $\hat C$ is smooth (and hence isomorphic to $\pr^1$) by blowing up if necessary. Thus it suffices to show that $\int_{\hat C} \scA = \int_C \scA$. This follows when $\scA$ is rational since the family of curves $\beta_i(t).C$ is flat and hence the degree of line bundles is preserved, and hence follows for general $\scA$ by an approximation argument, using that $\Pic(\pr^1)$ is one dimensional.

 \begin{remark} Stoppa-Sz\'ekelyhidi prove a result corresponding to Proposition \ref{bodylemmachow} in the projective case using finite dimensional geometric invariant theory \cite{JS,StSz}. As this is not available to us, our argument is very different. \end{remark} 
We now turn to the calculation of the norms and inner products on blowups. For this we rely on an observation of Sz\'ekelyhidi \cite[Proposition 37]{GS6}, that what we require holds when $\X_0$ is smooth, so that one is blowing up a (smooth) point on $\X_0$. Sz\'ekelyhidi's observation is that the computation is local around the point being blown up, hence follows from the corresponding result for projective space (or any smooth projective variety), which can be proven relatively easily using the algebro-geometric methods of \cite{StSz}.

\begin{proposition}Suppose $\X_0$ is a simple normal crossings divisor. The inner products satisfies $$\langle \alpha,\beta_i\rangle_{(\scY,\scA-\epsilon [E])} = \langle\alpha,\beta_i \rangle_{(\scX,\scA)} +  O(\epsilon^n).$$ A similar formula holds for $\langle \beta_i, \beta_i\rangle.$ \end{proposition}

\begin{proof} The proof for the inner product is the same as for the norm, so we work with the norm for notational simplicity. Let $q\in \X_0$ be the specialisation of $p\in X\cong\X_1$. The norm is calculated as an integral over $\X_0$, so let $\X = \sum a_j\X_{0,j}$ as a cycle, so that $\X_{0,j}$ are compact K\"ahler manifolds. 

Note that in general the blow-up of $\X_0$ at the point $q$ is not necessarily equal to $\Y_0$, for example if $\X_0$ is the intersection of two lines and $q$ is the point of intersection, $\Y_0$ will have an extra component. However, each component of $\Y_0$ is either the blow-up of a component of $\X_0$ at $q$ or a component $\Y_{0,j}$ such that the $\C^*$-action fixes each point on $\Y_{0,j}$. Hence the hamiltonian is trivial on such $\Y_{0,j}$, and so these do not affect the norm. 

Setting $\Y_{0,j}$ to be a component of $\Y_0$ which is the blow-up of $\X_{0,j}$ at $q$, it suffices to compare $$\int_{\X_{0,j}} |h_{\alpha} - \hat h_{\alpha}|^2\co_0^n$$ and $$\int_{\Y_{0,j}} |h_{\Y_{0,j},\alpha} - \hat h_{\Y_{0,j},\alpha}|^2(\co_0-\epsilon \nu_0)^n,$$ where we have picked $\nu\in [E]$ such that $\co-\epsilon \nu$ is relatively positive and set $h_{Y,\alpha}$ to be the induced hamiltonian. But the required comparison follows directly from the observation of Sz\'ekelyhidi \cite[Proposition 37]{GS6}.
\end{proof}

\begin{remark}\label{intersections-analytic} It seems very likely that one could prove this directly, without appealing to any projective results, by using equivariant Chern-Weil theory in the style of Atiyah-Bott \cite[Section 6]{AB2}. In the projective case Donaldson shows that from the $n$-dimensional polarised scheme $(\X_0,\scL_0)$ admitting two $\C^*$-actions, using a fibre bundle construction one can frm an $n+2$ dimensional scheme $(\scP,\scH)$ such that the inner product calculated on $(\X_0,\scL_0)$ becomes an intersection number on $(\scP,\scH)$ \cite[Section 5.1]{SD}. Mirroring this construction in our K\"ahler setting one can produce an analytic space, say $\scP$, with a $(1,1)$-form $\zeta$ such that $$\int_{\scP} \zeta^{n+2} = \int_{\X_0} h_\alpha h_{\beta} \omega_0^n,$$ so the norm and inner product are computed as integrals on $\scP$. The key point in the projective case is that $\zeta$ is closed, which Donaldson shows by demonstrating it is the curvature of a metric on the corresponding line bundle. Suppose for the moment one knew in the K\"ahler setting that $\zeta$ is closed. Then with $\Y_0$ the central fibre of the blow-up test configuration, the comparison of the inner products and norms simply becomes a comparison of intersection numbers on $\scP$ and the corresponding analytic space for $\Y_0$, which would then follow by standard arguments involving the Poincar\'e-Lelong formula. Nevertheless the simplest proof of this is the one we presented above, borrowing the projective result.

 \end{remark}

\bibliography{relative}
\bibliographystyle{amsplain}

\vspace{4mm}

\end{document}